\documentclass[preprint,11pt,sort&compress,numbers,draft]{elsarticle}

\usepackage{amsmath,amsthm,amsfonts,amssymb,latexsym,mathrsfs,color,cases,url,enumerate}
\usepackage{tikz}

\newtheorem{thm}{Theorem}[section]
\newtheorem{lem}[thm]{Lemma}

\newtheorem{prop}[thm]{Proposition}

\theoremstyle{definition}
\newtheorem{definition}[thm]{Definition}

\newtheorem{rem}[thm]{Remark}

\numberwithin{equation}{section}

\begin{document}
	
	\begin{frontmatter}
		
		\title{The analytic properties of Hoggatt triangles}
		\author{Jianxi Mao}
		\ead{maojx@dlut.edu.cn}
		\author{Wenle Shi\corref{cor2}}
		\ead{shi-wenle@hotmail.com}
		\cortext[cor2]{Corresponding author.}
		\address{School of Mathematical Sciences, Dalian University of Technology, Dalian 116024, P.R. China}
		\date{}

		\begin{abstract}
			The $d$-Hoggatt triangle is a lower triangular matrix
			whose entries are given by specific minors of Pascal's triangle formed by consecutive $d$ rows and $d$ columns.
			The cases $d=1,2,3$ correspond to Pascal's triangle, the Narayana triangle, and the Baxter triangle, respectively.
			In this paper,
			we present the infinite log-concavity of the row and column sequences, the log-concavity of the sequences along transversals, and the eventual log-convexity of the sequences along rays of the $d$-Hoggatt triangle.
			In addition, we prove the asymptotic normality of the row sequences and total positivity of the $d$-Hoggatt triangle.
		\end{abstract}
		

	\end{frontmatter}
	
	\section{Introduction}
	
	Baxter numbers enumerate Baxter permutations originating in
	Baxter's study~\cite{Bax64} of fixed points for the composite of commuting functions.
	By inventing a generating tree for Baxter permutations,
	Chung et al.~\cite{CGHK78} presented an explicit expression of the Baxter numbers $B_n$.
	Hoggatt conjectured a simple algorithm to generate the refined Baxter numbers
	$$
	B(n,k)=\frac{\binom{n}{k} \binom{n+1}{k} \binom{n+2}{k}}{\binom{k+1}{k} \binom{k+2}{k}},
	$$
	where $B_{n}=\sum_{k=0}^{n} B(n,k)$.
	Fielder and Alford~\cite{FA89} proved Hoggatt's conjecture
	and defined Hoggatt triangles with index $d$ as a generalization of Pascal's triangle. 
	The Baxter triangle is the case $d=3$.
	It was shown  that the entries of the Hoggatt triangle are special minors of Pascal's triangle,
    see~\cite{LMR21, Cig21}.
	\begin{definition}
		Let $d\ge 1$.
		Define the {\it $d$-Hoggatt triangle} $H_d=\left[H_d(n,k)\right]_{n,k\ge 0}$ as
		$$
		H_d(n,k)=\det\left(\binom{n+i}{k+j}\right)_{i,j=0}^{d-1}.
		$$
	\end{definition}

	Note that when $d=1,2,3$, the $d$-Hoggatt triangles correspond to
	Pascal's triangle, the Narayana triangle and the Baxter triangle $\left[B(n,k)\right]_{n,k\ge 0}$, respectively.
	The Hoggatt triangle and the Hoggatt numbers $H_d(n,k)$ have many fascinating properties~\cite{LMR21,Sta71, Sun08, YUAN}.
	From an analytic perspective, 
	Dilks~\cite{Dil15} proved the row generating functions of Hoggatt triangles have only real zeros.
	Recently, Fang, Zhang and Zhao~\cite{FZZ25} established the interlacing property of the row generating functions of the Baxter triangle, a property stronger than polynomial with only real zeros. 
	They also generalized this result to Hoggatt triangles,
	as a stronger version of Dilks' result.
	On the other hand,
	the analytic properties of Pascal's triangle have been extensively studied in the literature~\cite{LWW23,PWW23}.
    For example, Su and Wang~\cite{SW08} and Yu~\cite{Yu09} presented the log-behavior of 
    the sequences located in a transversal or a ray of Pascal's triangle.
    Zhu~\cite{Zhu19} showed the Stieltjes moment property of the ray of Pascal's triangle.
    However, the log-behavior of sequences in the Narayana triangle and the Baxter triangle 
    has been less studied than that of Pascal's triangle, see~\cite{CYW10,FZZ25}.
    The aim of this paper is to extend several analytic properties of Pascal's triangle to the Narayana triangle and the Baxter triangle.
    We will show that the Hoggatt triangles play an important unifying role in proving these results.

	Let $\alpha=\left\lbrace a_k \right\rbrace _{k=0}^{\infty}$
	be a sequence of nonnegative numbers
	with no internal zeros.
	The sequence is said to be {\it log-concave} (resp. {\it log-convex})
	if $a_k^2 \geq a_{k-1}a_{k+1}$ (resp. $a_k^2 \leq a_{k-1}a_{k+1}$) for all $k\geq 1$.
	The sequence is said to be {\it unimodal} if there exists an index $n$
	such that $a_0\leq a_1\leq \dots \leq a_{n-1} \leq a_{n} \geq a_{n+1} \geq\cdots$ or it is an increasing or decreasing sequence.
	Clearly, log-concavity implies unimodality.
	Define the operator $\mathcal{L}$ on the sequence $\alpha=\left\lbrace a_k \right\rbrace _{k=0}^{\infty}$ by
	$$\mathcal{L}(\alpha)=\left\{a_0^2, a_1^2-a_0a_2,a_2^2-a_1a_3,a_3^2-a_2a_4,... \right\}.$$
	Hence a sequence $\alpha$ is log-concave if and only if $\mathcal{L}(\alpha)$ is nonnegative.
	A sequence $\alpha$ is said to be $k$-log-concave if $\mathcal{L}^j(\alpha)$ is a nonnegative sequence for all $0\leq j \leq k$,
	and {\it infinitely log-concave} if it is $k$-log-concave for all $k \geq 0$.

	Following Karlin \cite{Kar68},
	a finite or infinite matrix is called {\it totally positive} if all its minors are nonnegative.
	The matrix is called totally positive of order $r$
	if its minors of all orders $\leq r$ are nonnegative.
	It is easy to see that a nonnegative sequence $\left\lbrace a_k \right\rbrace _{k=0}^{\infty}$
	with no internal zeros is log-concave
	if and only if its Toeplitz matrix $[a_{i-j}]_{i,j\ge0}$
	is totally positive of order~2.
	An infinite nonnegative sequence $\left\lbrace a_k \right\rbrace _{k=0}^{\infty}$
	is called a {\it P\'olya frequency sequence} (PF sequence for short)
	if its Toeplitz matrix $[a_{i-j}]_{i,j\ge0}$ is totally positive.
	Thus, a  PF sequence is log-concave and hence unimodal.
	We say that a finite sequence $a_0,a_1,\ldots,a_n$ is a PF sequence
	if the corresponding infinite sequence $a_0,a_1,\ldots,a_n,0,0,\ldots$ is a PF sequence.
	Call $\left\lbrace a_k \right\rbrace _{k=0}^{\infty}\in \mathrm{PF}$ if $\left\lbrace a_k \right\rbrace _{k=0}^{\infty}$ is a PF sequence.
	A fundamental characterization for P\'olya frequency sequences states that a sequence
	$\left\lbrace a_k\right\rbrace _{k= 0}^{\infty}$ is a PF sequence
	if and only if its generating function is analytic in a neighborhood of the origin and has the expansion
	\begin{equation}\label{pf-c}
		\sum_{k\ge 0}a_kx^k=cx^me^{\sigma x} \frac{\prod_{j\ge0} (1+\alpha_j x)}{\prod_{j\ge0} (1-\beta_j x)},
	\end{equation}
	where $c>0, m\in\mathbb{N}, \alpha_j,\beta_j,\sigma\ge 0$, and $\sum_{j\ge0} (\alpha_j+\beta_j)<+\infty$
	(see, for example, \cite[p. 412]{Kar68}).
	Aissen, Schoenberg and Whitney \cite{ASW52} showed that a finite nonnegative sequence $\left\lbrace a_k \right\rbrace _{k=0}^{n}\in \mathrm{PF}$
	if and only if the generating function of $\left\lbrace a_k \right\rbrace _{k=0}^{n}$ has only real zeros.

	A sequence $\alpha=\left\lbrace a_k \right\rbrace _{k=0}^{\infty}$ of nonnegative numbers is log-convex
	if and only if its Hankel matrix $[a_{i+j}]_{i,j\ge0}$
	is totally positive of order~2.
	Log-convex sequences are closely related to the Stieltjes moment sequence.
	We say that $\alpha=\left\lbrace a_k\right\rbrace _{k= 0}^{\infty}$ is a {\it Stieltjes moment sequence} (SM sequence for short)
	if it has the form
	$$
	a_k=\int_0^{\infty} x^k d{\mu(x)},
	$$
	where $\mu$ is a nonnegative measure on $[0,+\infty)$.
	The Stieltjes moment problem is one of the classical moment problems and arises naturally in many branches of mathematics~\cite{ST43,Sch17}.
	It is well known~\cite{Pin10,Sch17} that the following are equivalent:
	\begin{itemize}
		\item[(i)] $\alpha$ is a Stieltjes moment sequence.
		\item[(ii)] The Hankel matrix $[a_{i+j}]_{i,j\ge 0}$ is totally positive.
		\item[(iii)] For any $n\ge 0$, both $[a_{i+j}]_{0\le i,j\le n}$ and $[a_{i+j+1}]_{0\le i,j\le n}$ are positive semidefinite.
	\end{itemize}

	
	
	The organization of this paper is as follows.
	In Section 2, we show that for each $d$,
	the row and column sequences in the $d$-Hoggatt triangle are infinitely log-concave.
	In Section 3, we show that the sequences located in a transversal of the $d$-Hoggatt triangle are log-concave
	and the sequences located in a ray are asymptotically log-convex.
	In addition, we provide a criterion for determining when the sequences located in a ray form an SM sequence.
	In Sections 4 and 5, we prove the asymptotic normality and total positivity in Hoggatt triangles.
	In the last section, we consider the log-convexity of the row sums 
	and present a couple of problems.

	\section{Infinite log-concavity of row and column sequences}
	
	Let $\left\langle n \right\rangle _d=\binom{n+d-1}{d}$
	and $\left\langle n \right\rangle _d!=\prod_{j=1}^{n}\left\langle j \right\rangle _d$.
	Cigler~\cite{Cig21} presented the explicit expression of the $d$-Hoggatt number,
	\begin{equation}\label{defentry}
		H_d(n,k)=\frac{ \left\langle n \right\rangle _d!}{\left\langle k \right\rangle _d!~\left\langle n-k \right\rangle _d!}=\prod_{j=0}^{d-1}\frac{\binom{n+j}{k}}{\binom{k+j}{k}}, \quad 0\leq k \leq n
	\end{equation}
	and $H_d(n,k)=0$ for $k>n$.
	Using the theory of the multiplier sequence, Dilks derived the following result.
	\begin{lem}[{\cite[Theorem 3.8]{Dil15}}]\label{Dilks}
		The row generating function of the $d$-Hoggatt triangle
		\begin{equation*}
			H_{n,d}(x)=\sum_{k=0}^nH_d(n,k)x^k
		\end{equation*}
		has only real zeros for each $d\geq1$ and $n\geq 0$.
	\end{lem}

	Now we turn to the infinite log-concavity of the column sequences of the $d$-Hoggatt triangle.
	Let $\{a_k\}_{k=0}^{\infty}$ be a PF sequence interpolated by the polynomial $p(x)$,
	i.e., $a_k=p(k)$ for $k\ge 0$.
	Then the generating function of $a_k$ is of the form
	$$\sum_{k=0}^{\infty} a_k x^k=\frac{w(x)}{(1-x)^{d+1}},$$
	where $w(x)$ is a polynomial of degree at most $d$,
	see~\cite[Chapter 4.3]{Sta97}.
	Since $\{a_k\}_{k=0}^{\infty}$ is a PF sequence,
	by~\eqref{pf-c}, $w(x)$ has nonnegative coefficients and only real zeros.
   	Br\"and\'en~\cite{Bra11} and Br\"and\'en and Chasse~\cite{BC15} presented the infinite log-concavity of a PF sequence separately.

	\begin{thm}\label{pbmc}
		Let $\left\{a_k\right\}_{k=0}^{\infty}$ be a P\'olya frequency sequence.
		\begin{enumerate}
			\item [\rm(i)] If the polynomial $\sum_{k=0}^{n} a_k x^k$ with nonnegative coefficients has only real zeros,
			then the sequence $a_0, a_1,\dots,a_n$ is infinitely log-concave.
			\item [\rm(ii)] If there exist $p(x)\in \mathbb{R}[x]$ and $p(-2)=p(-1)=0$ such that
			$a_k=p(k)$ for each $k=0,1,2,\ldots$, 
			then $\left\{a_k\right\}_{k=0}^{\infty}$ is infinitely log-concave.
		\end{enumerate}
	\end{thm}
	The column generating function of Pascal's triangle is
	$$\sum_{n=0}^{\infty} \binom{n+k}{k} x^n=\frac{1}{(1-x)^{k+1}}.$$
	This implies that $\left\{\binom{n+k}{k}\right\}_{n=0}^{\infty}$
	is a PF sequence for each $k$.
	The interpolated polynomial
	$$p(x)=\binom{x+k}{k}:=\frac{(x+k)(x+k-1)\dots(x+1)}{k!}$$
	with $p(-2)=p(-1)=0$ for $k\ge 2$.
	As a corollary,
	Br\"and\'en and Chasse~\cite{BC15} proved that for each $k\ge 0$ and $i\in \mathbb{N}$, $\mathcal{L}^i\left(\left\{\binom{n+k}{k} \right\}_{n=0}^{\infty}\right)$ is a PF sequence
	and therefore the column sequence of Pascal's triangle is infinitely log-concave (the cases $k=0$ and $k=1$ trivially hold).
    Now we show the infinite log-concavity in the $d$-Hoggatt triangle.
	
	\begin{thm}\label{lcthm}
		The row sequence $\left\{H_d(n,k)\right\}_{k=0}^{n}$ and the column sequence of the $d$-Hoggatt triangle
		$\left\{H_d(n+k,k)\right\}_{n=0}^{\infty}$ are both infinitely log-concave for each $d\geq 1$ and $n, k\geq 0$.
	\end{thm}
	\begin{proof}
		Since the row generating function of $d$-Hoggatt triangle $H_{n,d}(x)$ has only real zeros,
        then the infinite log-concavity of the row sequence follows from Theorem~\ref{pbmc} (i).
        
		By~\eqref{defentry},
		$H_d(n,0)=1$ and
		$H_d(n+1,1)=\binom{n+d}{d}.$
		The case $k=0$ trivially holds
		and the case $k=1$ reduces to Br\"and\'en and Chasse's result for Pascal's triangle.
		For $k\ge 2$,
		we now show that 
		$\left\{H_d(n+k,k)\right\}_{n=0}^{\infty}$ is a PF sequence.
		Cigler~\cite[Theorem 4]{Cig21} showed that the column generating function of the $d$-Hoggatt triangle has the form
		$$\sum\limits_{n=0}^{\infty}H_d(n+k,k)x^n
		=\frac{\displaystyle\sum\limits_{j=0}^{(d-1)(k-1)}N(d,k,j)x^j}{(1-x)^{dk+1}},$$
		where $N(d,k,j)$
		are the Narayana numbers of dimension $d$ introduced by Sulanke~\cite{Sul04}.
		Chen, Yang and Zhang~\cite[Theorem 3.1]{CYZ18} showed that
		$$
		\displaystyle\sum\limits_{j=0}^{(d-1)(k-1)}N(d,k,j)x^j
		$$
		has only real nonpositive zeros,
		which implies that $\left\lbrace H_d(n+k,k) \right\rbrace_{n=0}^{\infty}$ is a PF sequence.
		
     	It is easy to see that $\left\{H_d(n+k,k)\right\}_{n=0}^{\infty}$
		is interpolated by the polynomial
		$$
		p(x)=\prod_{j=0}^{d-1}\frac{\binom{x+k+j}{k}}{\binom{k+j}{k}}
		=\prod_{j=0}^{d-1}\frac{(x+k+j)(x+k+j-1)\dots(x+j+1)}{(k+j)(k+j-1)\dots(j+1)}.
		$$
		Clearly, $p(-2)=p(-1)=0$.
		By Theorem~\ref{pbmc} (ii),
		$\left\{H_d(n+k,k)\right\}_{n=0}^{\infty}$ is infinitely log-concave.		
	\end{proof}

	\section{Log-behavior of transversal and ray sequences of Hoggatt triangles}

	Su and Wang~\cite{SW08} demonstrated that
	the sequence located in a transversal of Pascal's triangle is log-concave,
	and the sequence  located in a ray of Pascal's triangle, not parallel to the column or the diagonal, is {\it asymptotically log-convex},
	i.e. there exists a nonnegative number $m$ such that $a_m,a_{m+1}, a_{m+2},... $ is log-convex.
	Yu~\cite{Yu09} strengthened these results, 
	confirming the conjectures of Su and Wang~\cite{SW08}.
	Specifically,
	Yu established that the transversal sequences of Pascal's triangle are all PF sequences (see Theorem~\ref{Conj-Yu});
	and for the ray sequences of Pascal's triangle, 
	Yu provided an explicit expression for the index 
	$m$, ensuring that the sequence $\{a_i\}_{i=m}^{\infty}$ is log-convex.
	Zhu~\cite[Proposition 3.7]{Zhu19} further extended Yu's work 
	and showed that the sequence $\{a_i\}_{i=m}^{\infty}$ is not merely log-convex but also an SM sequence.

	\begin{lem}[{\cite[Theorem 1.5]{Zhu19}}]\label{pos-sem}
		Let $n_0 \geq k_0 \geq 0$ and $\{ p,q,d_j,e_j \} \subseteq \mathbb{R}^+$
		for $j \in \mathbb{N}$.
		Define the sequence
		$$b_i = \frac{\Gamma{(n_0+ip+1)}}{\Gamma{(k_0 +iq+1)}\Gamma{((n_0-k_0)+i(p-q)+1)}}
		\prod_{j=0}^{\ell}\frac{1}{id_j+e_j},~~~
		i=0,1,2,...,$$
		where $\Gamma(x)$ is the Gamma function.
		If $p>q$ and $(n_0+1)q/p-1 \leq k_0 \leq (n_0+1)q/p$,
		then $\{ b_i \}_{i=0}^{\infty}$ forms a Stieltjes moment sequence.
	\end{lem}

	We generalize these results to the Hoggatt triangles.
	
	\begin{thm}\label{ray}
		Let $n_0, k_0, p, q$ be nonnegative integers with $n_0\geq k_0$
		and $d\geq 1$.
		Define 
		$$C_i= H_d(n_0+ip,k_0+iq),\quad i=0,1,2,\ldots.$$
		Then the following statements hold.
		\begin{enumerate}
			\item [\rm(i)] If $p<q$, then $\left\{C_i \right\}_{i=0}^{\infty}$ is log-concave and therefore unimodal;
			\item [\rm(ii)] If $p>q>0$,  then $\left\{C_i \right\}_{i=0}^{\infty}$ is asymptotically log-convex;
			\item [\rm(iii)]
			If $p>q>0$, $(d-1)q\le p$ and
			$$\frac{(n_0+d)q}{p}-1 \leq k_0 \leq \frac{(n_0+1)q}{p},$$
			then $\{ C_i \}_{i=0}^{\infty}$ forms a Stieltjes moment sequence.
		\end{enumerate}
	\end{thm}
	\begin{proof}
		{(i)} Let
	$$	
		\left\langle
		\begin{matrix}
			n\\k
		\end{matrix}
		\right\rangle_d
		:=H_d(n,k)=\frac{\langle n\rangle_d!}
		{\langle k\rangle_d!\,\langle n-k\rangle_d!}.
		$$
		To prove that $\left\lbrace C_i \right\rbrace _{i=0}^{\infty}$ is log-concave,
		it suffices to prove that
		$$H_d(n+p,k+q)H_d(n-p,k-q) \leq H_d(n,k)^2$$
		for $n\geq k$ and $p<q$.
	    Claim that 
	    \begin{align*}
	    	 H_d(n+p,k+q)H_d(n-p,k-q)=&H_d(n+p,n-k)H_d(n-p,n-k)\\
	    	 &\times
	    	\frac{H_d(n-k,q-p)}{H_d(n-k+q-p,q-p)}
	    	\frac{H_d(k-p,q-p)}{H_d(k+q,q-p)}.
	    \end{align*}
	    This follows from
	    \begin{equation*}
	    	\begin{aligned}
	    		\left\langle \begin{matrix}
	    			n+p\\
	    			k+q
	    		\end{matrix}\right\rangle _d
	    		\left\langle \begin{matrix}
	    			n-p\\
	    			k-q
	    		\end{matrix}\right\rangle _d
	    		=&
	    		\frac{\left\langle n+p  \right\rangle _d!   \left\langle n-p  \right\rangle _d!}
	    		{\left\langle k+q  \right\rangle _d! \left\langle n+p-k-q  \right\rangle _d! \left\langle k-q  \right\rangle _d! \left\langle n-p-k+q \right\rangle _d!}\\
	    		=&\left\langle \begin{matrix}
	    			n+p\\
	    			n-k
	    		\end{matrix}\right\rangle _d
	    		\left\langle \begin{matrix}
	    			n-p\\
	    			n-k
	    		\end{matrix}\right\rangle _d\frac{\left\langle \begin{matrix}
	    				n-k\\
	    				q-p
	    			\end{matrix}\right\rangle _d}{\left\langle \begin{matrix}
	    				n-k+q-p\\
	    				q-p
	    			\end{matrix}\right\rangle _d}
	    		\frac{\left\langle \begin{matrix}
	    				k-p\\
	    				q-p
	    			\end{matrix}\right\rangle _d}{\left\langle \begin{matrix}
	    				k+q\\
	    				q-p
	    			\end{matrix}\right\rangle _d}.
	    	\end{aligned}
	    \end{equation*}
	    Combining the log-concavity of the column sequence (Theorem~\ref{lcthm} (ii))
	    and the symmetry $H_d(n,k)=H_d(n,n-k)$,
	    we have
	    $$H_d(n+p,n-k)H_d(n-p,n-k)\le H_d(n,n-k)^2=H_d(n,k)^2.$$
	    Moreover, the column sequence is increasing, since
	    $$
	    {\left\langle \begin{matrix}
	    		n+1\\
	    		k
	    	\end{matrix}\right\rangle _d} \bigg/ {\left\langle \begin{matrix}
	    		n\\
	    		k
	    	\end{matrix}\right\rangle _d}=
	    \frac{ \left\langle n+1 \right\rangle _d}{\left\langle n+1-k \right\rangle _d}\ge 1.
	    $$
	   Thus we obtain that
	   $$
	   \frac{H_d(n-k,q-p)}{H_d(n-k+q-p,q-p)}
	   \frac{H_d(k-p,q-p)}{H_d(k+q,q-p)}\le 1.
	   $$
	Hence,
		$$H_d(n+p,k+q)H_d(n-p,k-q)\leq H_d(n,k)^2.$$

		(ii) Let $n=n_0+ip,$ and $k=k_0+iq$.	It suffices to show that
		$$
		\Delta(k):=
		\frac{H_d(n+p,k+q)H_d(n-p,k-q)}{H_d(n,k)^2}\ge 1
		$$
		for all sufficiently large $k$, where $n=n_0+p\frac{k-k_0}{q}$ and $p>q>0$.
		By~\eqref{defentry},
		$$
		H_d(n,k)
		=\prod_{j=0}^{d-1}\frac{\binom{n+j}{k}}{\binom{k+j}{k}}=
		\prod_{j=0}^{d-1}
		\frac{(n+j)!j!}{(k+j)!(n-k+j)!}.
		$$
		Hence
		$$
		\Delta(k)
		=
		\prod_{j=0}^{d-1}\Delta_j(k),
		$$
		where
		\begin{align*}
		\Delta_j(k)=&\frac{(n+p+j)!(n-p+j)!}{(n+j)!^2} \frac{(k+j)!^2}{(k+q+j)!(k-q+j)!} \frac{(n-k+j)!^2}{(n-k+p-q+j)!(n-k-p+q+j)!}\\
		=&\prod_{r=1}^{p}
		\frac{n+j+r}{n-p+j+r}
		\prod_{r=1}^{q}
		\frac{k-q+j+r}{k+j+r}
		\prod_{r=1}^{p-q}
		\frac{n-k-p+q+j+r}{n-k+j+r}.
		\end{align*}
   
		Since $n=n_0+p\frac{k-k_0}{q}$, taking logarithms yields
		$$
		\begin{aligned}
			\log\Delta_j(k)
			&=\sum_{r=1}^{p}
			\left[\log \left(\frac{p}{q}k+n_0-\frac{pk_0}{q}+j+r\right)-\log\left(\frac{p}{q}k+n_0-\frac{pk_0}{q}-p+j+r\right)\right]  \\
			+&\sum_{r=1}^{q}
			\left[\log(k-q+j+r)-\log(k+j+r)\right]  \\
			+&\sum_{r=1}^{p-q}
			\left[\log\left(\frac{p-q}{q}k+n_0-\frac{pk_0}{q}-p+q+j+r\right)-
			\log\left(\frac{p-q}{q}k+n_0-\frac{pk_0}{q}+j+r\right)\right].
		\end{aligned}
		$$
		We use the expansion
		$$
		\log(Ak+a)
		=
		\log(Ak)+\frac{a}{Ak}
		-\frac{a^2}{2A^2k^2}
		+O(k^{-3}),
		$$
		where $A>0$ and $a$ is independent of $k$. 
		Applying this expansion to the three sums above, the coefficient of $\log k$ vanishes
		and  the coefficient of $k^{-1}$ is
		\begin{align*}
		&\sum_{r=1}^{p} \left(\frac{q(n_0+j+r)-pk_0}{p}-\frac{q(n_0+j+r-p)-pk_0}{p}\right)
		+\sum_{r=1}^{q} \left((j+r-q)-(j+r)\right)\\
		+&\sum_{r=1}^{p-q} \left(\frac{q(n_0+j+r-p+q)-pk_0}{p-q}-\frac{q(n_0+j+r)-pk_0}{p-q}\right)\\
		=&pq-q^2-q(p-q)=0.
		\end{align*}
		The coefficient of 2$k^{-2}$ is
		\begin{align*}
		&\sum_{r=1}^{p} \left(-\frac{(q(n_0+j+r)-pk_0)^2}{p^2}+\frac{(q(n_0+j+r-p)-pk_0)^2}{p^2}\right)
		+\sum_{r=1}^{q} \left(-(j+r-q)^2+(j+r)^2\right)\\
		+&\sum_{r=1}^{p-q} \left(-\frac{(q(n_0+j+r-p+q)-pk_0)^2}{(p-q)^2}+\frac{(q(n_0+j+r)-pk_0)^2}{(p-q)^2}\right)\\
		=&\sum_{r=1}^{p}\frac{q(pq+2pk_0-2q(n_0+j+r))}{p}+\sum_{r=1}^{q} (2q(j+r)-q^2)\\
		+&\sum_{r=1}^{p-q} \frac{2q^2(n_0+j+r)-2pqk_0-q^2(p-q)}{p-q}\\
		=& q(2pk_0-2q(n_0+j)-q)+q^2(2j+1)+q(2q(n_0+j)-2pk_0+q)\\
		=&q^2(2j+1).
	\end{align*}
		Therefore
		$$
		\log\Delta_j(k)
		=
		\frac{q^2(2j+1)}{2k^2}
		+O(k^{-3})
		$$
		and
		$$
		\log\Delta(k)
		=
		\sum_{j=0}^{d-1}\log\Delta_j(k)
		=
		\frac{q^2}{2k^2}
		\sum_{j=0}^{d-1}(2j+1)
		+O(k^{-3})=\frac{d^2q^2}{2k^2}+O(k^{-3})
		$$
	    Then
		$$
		\log\Delta(k)
		=
		\frac{d^2q^2}{2k^2}
		+O(k^{-3})>0
		$$
		for all sufficiently large $k$. Hence
		$$
		\Delta(k)>1,
		$$
		which proves the eventual log-convexity.
		
		(iii) Recall that
		$$C_i= H_d(n_0+ip,k_0+iq)
		=\prod_{j=0}^{d-1}\frac{\binom{n_0+ip+j}{k_0+iq}}{\binom{k_0+iq+j}{k_0+iq}}=\prod_{j=0}^{d-1} C_{i}^j,$$
		where 	
		\begin{align*}
			C_{i}^j&=\frac{\binom{n_0+ip+j}{k_0+iq}}{\binom{k_0+iq+j}{k_0+iq}}
			=\frac{\binom{n_0+ip+j}{k_0+iq}}{\prod_{\ell=0}^{j-1}(k_0+iq+\ell+1)/j!}\\
			&=\frac{\Gamma(n_0+ip+j+1)}{\Gamma(k_0+iq+1) \Gamma((n_0-k_0)+i(p-q)+j+1)}
			\frac{1}{\prod_{\ell=0}^{j-1}(\frac{k_0}{\ell+1}+i\frac{q}{\ell+1}+1)}.
		\end{align*}
		Here by convention, the canonical product
		$$\prod_{\ell=0}^{j-1}\left(\frac{k_0}{\ell+1}+i\frac{q}{\ell+1}+1\right)$$
		reduces to 1 when $j=0$.
		By Lemma~\ref{pos-sem}, if $p>q>0$ and
		$$(n_0+j+1)q/p-1 \leq k_0 \leq (n_0+j+1)q/p,$$
		then $\left\{ C_i^j \right\}_{i=0}^{\infty}$ is an SM sequence for a fixed $j$.
		Since $(d-1)q\le p$,
		for $0\le j\le d-1$,
		when $p>q>0$ and
		$$(n_0+d)q/p-1 \leq k_0 \leq (n_0+1)q/p,$$
		each sequence $\left\{ C_i^j \right\}_{i=0}^{\infty}$ forms an SM sequence.
			
		Claim that if $\{a_k\}_{k=0}^{\infty}$ and $\{b_k\}_{k=0}^{\infty}$ are two Stieltjes moment sequences,
		then so is $\{a_kb_k\}_{k=0}^{\infty}$.
		The sequence $\{a_k \}_{k=0}^{\infty}$ (resp. $\{b_k \}_{k=0}^{\infty}$)
		is an SM sequence if and only if
		Hankel matrices $[a_{i+j}]_{i,j=0}^n$ and $[a_{i+j+1}]_{i,j=0}^n$ (resp. $[b_{i+j}]_{i,j=0}^n$ and $[b_{i+j+1}]_{i,j=0}^n$)
		are positive semidefinite.
		The famous Schur product theorem states that the Hadamard product of two positive semidefinite matrices
		is still positive semidefinite, see~\cite[Theorem 7.5.3]{HJ13}.
		Then $[a_{i+j}b_{i+j}]_{i,j=0}^n$ and $[a_{i+j+1}b_{i+j+1}]_{i,j=0}^n$ are still positive semidefinite.
		Thus $\{a_kb_k\}_{k=0}^{\infty}$ is an SM sequence.
        Hence,
		the Hadamard product of two SM sequences is still an SM sequence.
		Therefore, $\left\{ C_i \right\}_{i=0}^{\infty}$  is an SM sequence.
	\end{proof}
	
	\begin{rem}
		Note that in Theorem \ref{ray} (iii), 
		the inequality $(n_0+d)q/p-1\leq k_0\leq(n_0+1)q/p$ implies $(d-1)q\leq p$.
		It would be interesting  to study the SM properties of $\{C_i\}_{i=0}^{\infty}$ under the condition $(d-1)q>p$.
	\end{rem}

	\section{Asymptotic normality in Hoggatt triangles}
	
	The normal distribution consists of a symmetric bell-shaped curve
	with a downward-sloping curve on each side of the mode.
	Let $a(n,k)$ be a double-indexed sequence of nonnegative numbers.
	In combinatorics,
	when $a(n,k)$ is unimodal with respect to $k$ for sufficiently large $n$,
	it is natural to consider the asymptotic normality.
	
    Let $X_n$ be a random variable, and let
	$$p(n,k)=P(X_n=k)=\frac{a(n,k)}{\sum_ja(n,j)}$$
	denote the normalized probabilities.
	Following Bender~\cite{Ben73},
	we say that the sequence $a(n,k)$ (or equivalently, $X_n$) is {\it asymptotically normal by a central limit theorem},
	if
	\begin{equation}\label{clt}
		\lim_{n \rightarrow \infty}\sup_{x \in \mathbb{R}}\left| \sum_{k\leq \mu_n+x\sigma_n}p(n,k)-\frac{1}{\sqrt{2\pi}}\int_{-\infty}^{x}e^{\frac{-t^2}{2}}dt \right| =0,
	\end{equation}
	where $\mu_n$ and $\sigma_n^2$ are the mean and variance of $a(n,k)$, respectively. 

	The study of limit distributions of combinatorial sequences is useful for precisely
	predicting the properties of large structured combinatorial configurations.
	Many well-known combinatorial sequences enjoy central and local limit theorems.
	For example,
	the famous de Moivre-Laplace theorem states that
	the binomial coefficients $H_1(n,k)$ are asymptotically normal (by central and local limit theorems).
	We refer the reader to~\cite{Can15} for other examples.
	Chen et al.~\cite{CMW20} showed that the Narayana numbers $H_2(n,k)$ are asymptotically normal
	and Zhao~\cite{Zha24} showed that the refined Baxter numbers $H_3(n,k)$ are asymptotically normal.
	Both proofs rely on the fact that the row generating functions have only real zeros.
	Here we prove the general result that the $d$-Hoggatt numbers $H_d(n,k)$ are asymptotically normal
	by direct approximations.
	Throughout this section, 
	the symbols $O$, $o$ and $\sim$ have their usual meaning:
	
	$f(n)=O(g(n))$ means $f(n)/g(n)$ is bounded as $n\rightarrow\infty$;
	
	$f(n)=o(g(n))$ means $f(n)/g(n)\rightarrow0$ as $n\rightarrow\infty$;
	
	$f(n)\sim g(n)$ means $f(n)/g(n)\rightarrow1$ as $n\rightarrow\infty$.
	
	Suppose that  $a(n,k) \sim b(n,k)$.
	Note that it does not imply 
	$$
	\sum_k a(n,k) \sim \sum_k b(n,k).
	$$
	We say that $a(n,k) \sim b(n,k)$ {\it uniformly} for $k \in S(n)$ provided $a(n,k)=b(n,k)(1+o(1))$
	whenever $k \in S(n)$, where the constant in the $o$ is independent of $k$. 
	In this case we can write 
	\begin{equation*}\label{Asy-Sn}
		\sum_{k \in S(n)}a(n,k)=\sum_{k \in S(n)}b(n,k)+o\left(\sum_{k\in S(n)} |b(n,k)| \right)
	\end{equation*}
	and in particular 
	$$
	\sum_{k \in S(n)}a(n,k)\sim \sum_{k \in S(n)}b(n,k).
	$$
	See~\cite[p. 487]{Ben74} for instance.
	The following proof is motivated by Bender~\cite{Ben74}.

	\begin{thm}\label{asythm}
		For each $d\ge 1$,
		the $d$-Hoggatt numbers $H_d(n,k)$ are asymptotically normal,
		with the mean $\mu_n=\frac{n}{2}$ and variance $\sigma_n^2\sim\frac{n}{4d}$.
	\end{thm}

	\begin{proof}
		By the symmetry
		$$
		H_d(n,k)=H_d(n,n-k),
		$$
		we immediately have $\mu_n=\frac{n}{2}$.

		We first assume that $n=2m$. Let
		$$
		A_{m,d}(s):=H_d(2m,m+s),\qquad -m\le s\le m.
		$$
		By~\eqref{defentry},
		$$
		A_{m,d}(s)
		=\prod_{j=0}^{d-1}\frac{\binom{2m+j}{m+s}}{\binom{m+s+j}{m+s}}=
		\prod_{j=0}^{d-1}
		\frac{(2m+j)!j!}{(m+s+j)!(m-s+j)!}.
		$$
		Hence
		\begin{equation*}
		\frac{A_{m,d}(s)}{A_{m,d}(0)}
		=
		\prod_{j=0}^{d-1}
		\frac{(m+j)!^2}{(m+s+j)!(m-s+j)!}=\prod_{j=0}^{d-1}
		\prod_{r=1}^{s}\frac{m+j-r+1}{m+j+r}.
		\end{equation*}
		For $s\ge0$, taking logarithms yields
		\begin{equation}\label{estimate}
		\log\frac{A_{m,d}(s)}{A_{m,d}(0)}
		=
		-\sum_{j=0}^{d-1}\sum_{r=1}^s
		\log\frac{m+j+r}{m+j-r+1}.
		\end{equation}
		Recall that $$\log(x+r)=\log x+\frac{r}{x}-\frac{r^2}{2x^2}+O\left(\frac{r^3}{x^3}\right).$$
		For fixed $j$, $0\le j\le d-1,$ let
		$$
		x=m+j.
		$$ 
		Since $d$ is fixed, we have $x\sim m$. Moreover, for $1\le r\le s=o(m)$,
		the term in~\eqref{estimate} is
		$$
		\log\frac{x+r}{x-r+1}
		=
		\frac{2r-1}{x}
		+
		O\left(\frac{r}{x^2}+\frac{r^3}{x^3}\right),
		$$
		uniformly in $r$. Hence, for $1\le r\le s=o(m)$,
		$$
		\begin{aligned}
			\log\frac{A_{m,d}(s)}{A_{m,d}(0)}
			=
			-\sum_{j=0}^{d-1}\sum_{r=1}^{s}
			\left(
			\frac{2r-1}{m+j}
			+
			O\left(\frac{r}{(m+j)^2}+\frac{r^3}{(m+j)^3}\right)
			\right)
		\end{aligned}
		$$
		Since
		$$
		\sum_{r=1}^{s}(2r-1)=s^2,
		\qquad
		\sum_{r=1}^{s}r=O(s^2),
		\qquad \sum_{r=1}^{s}r^3=O(s^4),
		$$
		and $d$ is fixed, we get
		$$
		\log\frac{A_{m,d}(s)}{A_{m,d}(0)}
		=
		-s^2\sum_{j=0}^{d-1}\frac{1}{m+j}
		+
		O\left(\frac{s^2}{m^2}+\frac{s^4}{m^3}\right).
		$$
		By
		$$
		\sum_{j=0}^{d-1}\frac{1}{m+j}
		=
		\frac{d}{m}
		+
		O\left(\frac{1}{m^2}\right),
		$$
		we obtain that uniformly for $|s|\le m^{2/3}$,
		$$
		\log\frac{A_{m,d}(s)}{A_{m,d}(0)}
		=
		-\frac{ds^2}{m}
		+
		O\left(\frac{s^2}{m^2}\right)
		+
		O\left(\frac{s^2}{m^2}+\frac{s^4}{m^3}\right)=-\frac{ds^2}{m}
		+
	   O\left(\frac{s^2}{m^2}+\frac{s^4}{m^3}\right).
		$$
		In particular, 
		$$
		\log\frac{A_{m,d}(s)}{A_{m,d}(0)}
		=
		-\frac{d s^2}{m}
		+
	O\left(\frac{s^2}{m^2}+\frac{s^4}{m^3}\right).
		$$
		Since $n=2m$, we have
		\begin{equation}\label{keypoint}
		\frac{A_{m,d}(s)}{A_{m,d}(0)}
		=
		\exp\left(-\frac{2ds^2}{n}\right)
		\left(1+o(1)\right)
		\end{equation}
		uniformly for $|s|\le m^{2/3}$.
		In particular, 	this estimate holds uniformly for \textbf{$|s|=O(\sqrt n)$.}
		
		Now we show that the contribution of
		$|s|>m^{2/3}$ is negligible. Indeed, by~\eqref{defentry},
		$$
		\frac{A_{m,d}(s+1)}{A_{m,d}(s)}
		=
		\prod_{j=0}^{d-1}
		\frac{m-s+j}{m+s+j+1}
		$$
		for $0\le s<m$.
		Taking logarithms yields
		$$
		\log \frac{A_{m,d}(s+1)}{A_{m,d}(s)}=\sum_{j=0}^{d-1}\log\frac{m-s+j}{m+s+j+1}
		=\sum_{j=0}^{d-1}\log\left( 1-\frac{2s+1}{m+s+j+1}\right) .
		$$
		For $0\leq s \leq m/2$, 
		we get $0<\frac{2s+1}{m+s+j+1} <1$.
		Using the inequality $\log(1-z)\leq -z$ for $0<z<1$, 
		we have
		$$	
		\log \frac{A_{m,d}(s+1)}{A_{m,d}(s)}
		\leq -\sum_{j=0}^{d-1}\frac{2s+1}{m+s+j+1}
		\leq -\frac{2sd}{m+s+d}
		\leq -c_1\frac{s}{m}
		$$
		for some constant $c_1>0$ depending only on $d$.	
		Thus, for $0\le s\le m/2$, we have
		$$
		\frac{A_{m,d}(s+1)}{A_{m,d}(s)}
		\le
		\exp\left(-c_1\frac{s}{m}\right).
		$$
		 Iterating this bound yields
		$$
		A_{m,d}(s)\le A_{m,d}(0)\prod_{r=0}^{s-1}\exp\left(-c_1\frac{r}{m}\right)
		=A_{m,d}(0)\exp\left(-c_1\frac{s^2-s}{2m}\right)\le A_{m,d}(0)
		$$
		for some constant $c_1>0$.
		For \(s>m/2\), since \(A_{m,d}(s)\) is
		decreasing for \(s\geq0\), we also have
		\[
		A_{m,d}(s)
		\leq
		A_{m,d}(\lfloor m/2\rfloor)
		\leq
		A_{m,d}(0)e^{-c_2m}
		\]
		for some constant $c_2>0$.
		By symmetry, the same estimates hold for negative \(s\). Consequently,
		\begin{align*}
			\sum_{|s|>m^{2/3}} A_{m,d}(s)=&\sum_{m^{2/3}\le |s|\le m/2} A_{m,d}(s)+\sum_{m/2< |s|\le m} A_{m,d}(s)\\
		    \le &  m A_{m,d}(0)\exp\left(-c_1\frac{(m^{2/3})^2-m^{2/3}}{2m}\right)+mA_{m,d}(0)e^{-c_2m}\\
			=&m A_{m,d}(0) e^{-c_3m^{1/3}}+mA_{m,d}(0)e^{-c_2m}=o\bigl(A_{m,d}(0)\bigr)
		\end{align*}

		Therefore, by~\eqref{keypoint}
		\begin{align*}
		Z_{n,d}&:=\sum_{k=0}^n H_{d}(n,k)
		=\sum_{s=-m}^m A_{m,d}(s)\\
		&\sim\sum_{|s|\le m^{2/3}} A_{m,d}(s)\sim
		A_{m,d}(0)
		\sum_{s\in\mathbb Z}
		\exp\left(-\frac{2ds^2}{n}\right).
		\end{align*}
		Recall that $n=2m.$
		Using the standard Gaussian sum estimate
		$$
		\sum_{s\in\mathbb Z}e^{-2ds^2/n}
		\sim
		\sqrt{\frac{\pi n}{2d}},
		$$
		we get
		$$
		Z_{n,d}
		\sim
		A_{m,d}(0)\sqrt{\frac{\pi n}{2d}}.
		$$
		Consequently, uniformly for $s=O(\sqrt n)$,
		$$
		\frac{A_{m,d}(s)}{Z_{n,d}}
		\sim
		\sqrt{\frac{2d}{\pi n}}
		\exp\left(-\frac{2ds^2}{n}\right).
		$$
		Since $s=k-n/2$, this is exactly
		$$
		\frac{H_d(n,k)}{Z_{n,d}}
		\sim
		\frac{1}{\sqrt{2\pi n/(4d)}}
		\exp\left(-\frac{(k-n/2)^2}{2n/(4d)}\right).
		$$

		Let
		\[
		\tau_n=\sqrt{\frac{n}{4d}} .
		\]
		For fixed \(x\in\mathbb R\), by the above estimate and the tail estimate,
		\[
		\begin{aligned}
			P\left(\frac{X_n-n/2}{\tau_n}\le x\right)
			&=
			\sum_{s\le x\tau_n}\frac{A_{n/2,d}(s)}{Z_{n,d}}  \\
			&=
			\sum_{-(n/2)^{2/3}\le s\le x\tau_n}
			\frac{A_{n/2,d}(s)}{Z_{n,d}}
			+o(1)  \\
			&=
			\frac1{\tau_n}
			\sum_{-(n/2)^{2/3}\le s\le x\tau_n}
			\frac1{\sqrt{2\pi}}
			\exp\left(-\frac{s^2}{2\tau_n^2}\right)
			+o(1).
		\end{aligned}
		\]
		Consequently, uniformly for \(|s|\le (n/2)^{2/3}\),
		\[
		\frac{A_{n/2,d}(s)}{Z_{n,d}}
		=
		\frac1{\tau_n}
		\frac1{\sqrt{2\pi}}
		\exp\left(-\frac{s^2}{2\tau_n^2}\right)(1+o(1)),
		\]
		where \(\tau_n=\sqrt{n/(4d)}\).
		Since the last sum is a Riemann sum, we obtain
		\[
		\frac1{\tau_n}
		\sum_{-(n/2)^{2/3}\le s\le x\tau_n}
		\frac1{\sqrt{2\pi}}
		\exp\left(-\frac{s^2}{2\tau_n^2}\right)
		\longrightarrow
		\frac1{\sqrt{2\pi}}\int_{-\infty}^{x}e^{-t^2/2}\,dt .
		\]
		By the same uniform estimate in the central range and the above tail estimate, we have
		\[
  	  \sum_{s=-m}^m s^2 A_{m,d}(s)
  	\sim
	A_{m,d}(0)\sum_{s\in\mathbb Z}s^2
	\exp\left(-\frac{2ds^2}{n}\right).
		\]
		Together with
		\[
		Z_{n,d}
		\sim
		A_{m,d}(0)\sum_{s\in\mathbb Z}
		\exp\left(-\frac{2ds^2}{n}\right),
		\]
		it follows that
		\[
		\sigma_n^2
		\sim
		\frac{
			\sum_{s\in\mathbb Z}s^2 e^{-2ds^2/n}
		}{
			\sum_{s\in\mathbb Z}e^{-2ds^2/n}
		}
		\sim
		\frac{n}{4d}.
		\]
		Thus
		\[
		\sigma_n\sim \tau_n.
		\]
		Therefore
		\[
		\frac{X_n-n/2}{\sqrt{n/4d}}\longrightarrow N(0,1).
		\]
		Since the limiting distribution function is continuous, pointwise convergence of distribution functions implies uniform convergence in $x$.
		
		The case where $n$ is odd is treated in the same way by writing
		$k=n/2+s$; the half-integer shift changes the estimates only by $O(1)$ and does not affect
		the asymptotics. This completes the proof.
	\end{proof}

	\section{Total positivity of Hoggatt triangles}
	
	The purpose of this section is to prove the total positivity of the Hoggatt triangles.
	It is folklore~\cite{Pin10} that Pascal's triangle $[H_1(n,k)]_{n,k\ge 0}$ is totally positive.
	Wang and Yang~\cite{WY18} showed that the Narayana triangle $[H_2(n,k)]_{n,k\ge 0}$ and the Narayana square $[H_2(n+k,k)]_{n,k\ge 0}$ are totally positive.
	Inspired by Wang and Yang,
	we consider the total positivity of the $d$-Hoggatt triangle.
	
	A sequence $T=\{ \gamma_k \}_{k=0}^{\infty}$ of real numbers is called a {\it multiplier sequence} if, whenever the real polynomial $p(x)=\sum_{k=0}^{n}a_kx^k$ has only real zeros, the polynomial $T[p(x)]=\sum_{k=0}^{n}\gamma_ka_kx^k$ also has only real zeros. By the definition of multiplier sequence, we can obtain the following result.
	\begin{lem}\label{Hadamard-ms}
		Let $\left\lbrace  \gamma_k \right\rbrace _{k=0}^{\infty}$ and $\left\lbrace  \tau_k \right\rbrace _{k=0}^{\infty}$ be multiplier sequences.
		Then so is $\left\lbrace \gamma_k\tau_k \right\rbrace _{k=0}^{\infty}$.
	\end{lem}
	
	Let $\varphi(x)=\sum_{k=0}^{\infty}\frac{\gamma_k}{k!}x^k$.
	Then the sequence $\{\gamma_k\}_{k=0}^{\infty}$
	is a multiplier sequence if and only if
	either $\varphi(x)$ or $\varphi(-x)$ defines an entire function that can be written as
	\begin{equation}\label{ms-tran}
		cx^me^{\sigma x}\prod_{j=1}^{\infty}(1+\alpha_{j}x),
	\end{equation}
	where $c \in \mathbb{R}$, $m \in \mathbb{N}$, $\sigma \geq 0$, $\alpha_j \geq 0$ for all $j \in \mathbb{N}$, and $\sum_{j\geq1}\alpha_j<\infty$ (see \cite{Lev80} for instance).
	Comparing \eqref{ms-tran} with the generating function of PF sequences~\eqref{pf-c},
	the following result is immediate.
	\begin{lem}\label{MSPF}
		If a sequence $\left\lbrace \gamma_k \right\rbrace _{k=0}^{\infty}$ of nonnegative numbers is a multiplier sequence,
		then $\left\lbrace \gamma_k/k! \right\rbrace _{k=0}^{\infty}$ is a P\'olya frequency sequence.
	\end{lem}
	
	Let $A=[a_{n,k}]_{n,k\geq0}$ and $B=[b_{n,k}]_{n,k\geq 0}$ be two matrices.
	If there exist positive numbers $x_n$ and $y_k$ such that $b_{n,k}=x_ny_ka_{n,k}$ for all $n$ and $k$,
	then we denote $a_{n,k} \simeq b_{n,k}$ and $A \simeq B$.
	\begin{lem}[{\cite[Proposition 2.1]{WY18}}]\label{wyprop}
		Suppose that $A \simeq B$.
		Then the matrix $A$ is totally positive if and only if  the matrix $B$ is totally positive.
	\end{lem}
	
	We obtain the following results.

	\begin{thm}\label{tp}
		For each $d\ge1$, both the $d$-Hoggatt triangle
	$$
	H_d=[H_d(n,k)]_{n,k\ge0}
	$$
	and the $d$-Hoggatt square
	$$
	H_d^\ulcorner=[H_d(n+k,k)]_{n,k\ge0}
	$$
	are totally positive.
	\end{thm}
	\begin{proof}
	Note that $1/r!=0$ when $r<0$.	By  \eqref{defentry},
	$$
	H_d(n,k)
	=
	\prod_{j=0}^{d-1}
	\frac{\binom{n+j}{k}}{\binom{k+j}{k}}.
	$$
	Since $\binom{k+j}{k}$ depends only on $k$, multiplying the $k$-th column by the positive
	number $\prod_{j=0}^{d-1}\binom{k+j}{k}$ does not affect total positivity. Hence
	$$
	H_d(n,k)\simeq
	\prod_{j=0}^{d-1}\binom{n+j}{k}.
	$$
	Furthermore,
	$$
	\binom{n+j}{k}
	=
	\frac{(n+j)!}{k!(n-k+j)!}
	\simeq
	\frac{1}{(n-k+j)!},
	$$
	where the factor $(n+j)!$ depends only on $n$, and $k!$ depends only on $k$. Therefore
	$$
	H_d(n,k)
	\simeq
	\prod_{j=0}^{d-1}\frac{1}{(n-k+j)!}.
	$$
		Writing $N=n-k$, we have
	$$
	\prod_{j=0}^{d-1}\frac{1}{(N+j)!}
	\simeq
	\frac{1}{N!}\prod_{j=2}^{d}\frac{1}{(j)_N}.
	$$
	Thus the total positivity of $H_d$ is equivalent to that of the Toeplitz
	matrix generated by the sequence
	$$
	\alpha_n=
	\frac{1}{n!}\prod_{j=2}^{d}\frac{1}{(j)_n},
	\qquad n\ge0.
	$$
		A classic result~\cite[p.341]{Lev80} of Laguerre on multiplier sequences states that
		for any $t>0$, the sequence $\left\lbrace {1}/{(t)_n}\right\rbrace _{n=0}^{\infty}$ is a multiplier sequence.
		It follows by Lemma~\ref{Hadamard-ms} that the Hadamard product sequence 
		$$\left\lbrace \prod_{j=2}^{d}{1}/{(j)_{n}} \right\rbrace _{n= 0}^{\infty}$$ is a multiplier sequence.
		Thus, the sequence
		$\left\lbrace \frac{1}{n!}\prod_{j=2}^{d}\frac{1}{(j)_{n}} \right\rbrace _{n= 0}^{\infty}$
		is a PF sequence by Lemma~\ref{MSPF},
		which implies that the corresponding Toeplitz matrix $$\left[ \frac{1}{(n-k)!}\prod_{j=2}^{d}\frac{1}{(j)_{n-k}} \right] _{n,k\geq 0}$$ is TP.
		By Lemma~\ref{wyprop},
		$H_d=[H_d(n,k)]_{n,k\geq 0}$ is totally positive.

		Next we prove the total positivity of $H_d^\ulcorner$. Again by~\eqref{defentry},
		$$
		H_d(n+k,k)
		=
		\prod_{j=0}^{d-1}
		\frac{\binom{n+k+j}{k}}{\binom{k+j}{k}}.
		$$
		Removing row and column factors, we get
		$$
		H_d(n+k,k)
		\simeq
		\prod_{j=0}^{d-1}(n+k+j)!.
		$$
		The sequence $\{m!\}_{m=0}^{\infty}$ is an SM sequence since
		$
		m! = \int_0^\infty x^m e^{-x}\,dx.
		$
		The total positivity of $[(n+k)!]_{n,k\ge 0}$ implies that of $[(n+k+j)!]_{n,k\ge 0}$ for $j\ge 0$,
		since $[(n+k+j)!]_{n,k\ge 0}$ is a submatrix of $[(n+k)!]_{n,k\ge 0}$,
		and the submatrix of a totally positive matrix is still totally positive.
		Thus the sequence $\left\lbrace (k+j)!\right\rbrace _{k= 0}^{\infty}$ is an SM sequence.
		The Hadamard product of two SM sequences is still an SM sequence, 
		see the proof of Theorem~\ref{ray} (iii). 
		Hence,
		the Hankel matrix
		$$\left[\prod_{j=0}^{d-1}(n+k+j)!\right]_{n,k\ge 0}$$
		is totally positive.
		Thus we obtain that
		$H_d^{\ulcorner}=[H_d(n+k,k)]_{n,k\geq 0}$ is totally positive.
	\end{proof}
	
	\section{Further work}
	
	Let  
	$$
	H_{n,d}=H_{n,d}(1)=\sum_{k=0}^{n}H_d(n,k)
	$$
	be the $n$-th row sum of the $d$-Hoggatt triangle.
	It is easy to see that
	$H_{n,1}=2^n$ and $H_{n,2}=\frac{1}{n+2} \binom{2n+2}{n+1}$ are log-convex.
	Consider the Baxter numbers
	$$B_{n}=H_{n,3}=\sum_{k=0}^{n}H_3(n,k)=\sum_{k=0}^{n} \frac{\binom{n}{k} \binom{n+1}{k} \binom{n+2}{k}}
	{\binom{k+1}{k} \binom{k+2}{k}}.$$
	Bouvel et al.~\cite{BGRR18} investigated the combinatorial interpretations of the Baxter numbers $B_n$.
	It is natural to consider the log-convexity of the sequence of row sums of the Hoggatt triangles.
    Using the {\tt Mathematica} package fastZeil~\footnote{The package is accessible at https://risc.jku.at/sw/fastzeil/.} given by Paule and Schorn~\cite{PS95},
	we obtain the recurrence relation for the Baxter numbers $B_n$.
	
	\noindent
	{\tiny\sffamily In[1]:=}\
{\footnotesize\bfseries\boldmath $\ll$ RISC`fastZeil`;}

\noindent
{\tiny\sffamily In[2]:=}\,
{\small\bfseries\boldmath
	$\mathrm{Zb}[\mathrm{Binomial}[n,k]\mathrm{Binomial}[n+1,k]\mathrm{Binomial}[n+2,k]/
	(\mathrm{Binomial}[k+1,k]$
}

\hspace*{2em}
{\small\bfseries\boldmath
	$\mathrm{Binomial}[k+2,k]),\{k,0,n\},n,2]$
}

\vspace{2pt}

\noindent
{\tiny\sffamily Out[2]=}\,
{\small\bfseries
	$\{-8(1+n)(2+n)\mathrm{SUM}[n]+(-82-49n-7n^2)\mathrm{SUM}[1+n]$
	$+(5+n)(6+n)\mathrm{SUM}[2+n]==0\}$
}

	The output implies that $B_n$ satisfies the following recurrence relation:
	\begin{equation}\label{recb}
		\begin{aligned}
			(n+4)(n+5)B_{n+1} = (40+35n+7n^2)B_{n} + 8n(n+1)B_{n-1}
		\end{aligned}
	\end{equation}
	with $B_0=1,~B_1=2$, see \cite[A001181]{Slo24} for instance.
	Liu and Wang presented some criteria for the log-convexity
	of the sequences satisfying the
	three-term recurrence.
	\begin{thm}[{\cite[Theorem 3.1]{LW07}}]\label{liu-wang}
		Assume that $\{z_n\}_{n=0}^{\infty}$ satisfies
		\begin{equation*}
			\begin{aligned}
				a_nz_{n+1} = b_nz_n + c_nz_{n-1},
			\end{aligned}
		\end{equation*}
		where $a_n$, $b_n$ and $c_n$ are positive for $n \geq 1$.
		Let
		\begin{equation*}\label{lambda}
			\begin{aligned}
				\lambda_n = \frac{b_n+\sqrt{b^2_n+4a_nc_n}}
				{2a_n}.
			\end{aligned}
		\end{equation*}
		Suppose that the sequence $z_0$, $z_1$, $z_2$, $z_3$ is log-convex and that the inequality
		$$a_n\lambda_{n-1}\lambda_{n+1}-b_n\lambda_{n-1}-c_n \geq 0$$
		holds for $n \geq 2$.
		Then the sequence $\{z_n\}_{n=0}^{\infty}$ is log-convex.
	\end{thm}
	We obtain the log-convexity of the Baxter numbers.
	\begin{prop}
		The sequence of the Baxter numbers $\{B_n \}_{n=0}^{\infty}$ is log-convex.
	\end{prop}
	\begin{proof}
		According to Theorem \ref{liu-wang} and (\ref{recb}), we have
		\begin{equation*}
			\begin{aligned}
				\lambda_n = \frac{40+35n+7n^2+\sqrt{(40+35n+7n^2)^2+32n(n+1)(n+4)(n+5)}}
				{2(n+4)(n+5)}.\\
			\end{aligned}
		\end{equation*}
		For $n\geq 2$, with {\tt Mathematica}, we can prove that
		$$(n+4)(n+5)	\lambda_{n-1}	\lambda_{n+1}-(40+35n+7n^2)	\lambda_{n-1}-8n(n+1)\geq 0.$$
		
		\noindent
		{\tiny\sffamily In[3]:=}\,
		{\normalsize\bfseries\boldmath
			$\lambda[n\_]=$}	{\large\bfseries\boldmath$\frac{40+35n+7n^2+\sqrt{(40+35n+7n^2)^2+32n(n+1)(n+4)(n+5)}}{2(n+4)(n+5)};$
		}
		\vspace{2pt}
		
		\noindent
	{\tiny\sffamily In[4]:=}\,
	{\small\bfseries\boldmath
		$\mathrm{ForAll}[\{n\}, n \geq 2, (n+4)(n+5)\lambda[n-1]	\lambda[n+1]-
		(40+35n+7n^2)\lambda[n-1]-8n(n+1)\geq  \hspace*{2.5em} 0]// \mathrm{Resolve}$
	}
	\vspace{2pt}
	
			\noindent
		{\tiny\sffamily Out[4]=}
			{\small	True}
		
	Since $B_0=1, B_1=2, B_2=6, B_3=22$,
		it is clear that the sequence $B_0, B_1, B_2, B_3$ is log-convex.
		Hence the sequence $\{B_n\}_{n=0}^{\infty}$ is log-convex by Theorem~\ref{liu-wang}.
	\end{proof}
	
	Similarly, we derive that the row sum of the $4$-Hoggatt matrix
	$H_{n,4}=\sum_{k=0}^{n}H_4(n,k)$ satisfies the three-term recurrence
	\begin{equation*}
		\begin{aligned}
			(n+4)(n+5)(n+6)(n+7)H_{n+1,4} =& 6(n+2)(n+4)(n+5)(7+2n)H_{n,4}\\
			& +4n(n+1)(11+4n)(13+4n)H_{n-1,4}
		\end{aligned}
	\end{equation*}
	with $H_{0,4}=1, H_{1,4}=2$, see \cite[A005362]{Slo24} for instance.
	The log-convexity of $\{H_{n,4} \}_{n=0}^{\infty}$ can be obtained similarly.
	
	Unfortunately for $d\ge 5$,
	we have not found the three-term recurrence of the row sum.
	Numerical results suggest that the sequence $\{H_{n,d} \}_{n=0}^{\infty}$ is log-convex.
	Kot\v{e}\v{s}ovec~\cite{Kot21} considered the asymptotic formula for the numbers $H_{n,d}~(n\geq1)$,
	\begin{equation*}
		\begin{aligned}
			H_{n,d} \sim 	2^{-\frac{1}{2}+d(-\frac{1}{2}+d+n)}n^{-\frac{d^2}{2}}\sqrt{\frac{n}{d}}{\pi}^{\frac{1-d}{2}} G(1+d),
		\end{aligned}
	\end{equation*}		
	where $G(n)$ is the Barnes G-function.
	This shows that the sequence of $\{H_{n,d} \}_{n=0}^{\infty}$ is asymptotically log-convex for $d\geq 5$.
	Note that 
	\begin{equation}\label{explicit}
		2^n=\int_{0}^{\infty}x^n\delta(x-2)dx \quad\textrm{and}\quad \frac{1}{n+1}\binom{2n}{n}=\int_{0}^{\infty}x^n
		\left(\frac{1}{2\pi}\sqrt{\frac{4-x}{x}} \right)I_{(0,4]}(x)dx
	\end{equation}
	where $\delta(x)$ is the Dirac delta function and $I_A$ is the indicator function of the set $A$, see~\cite[Eq.(1.2)]{LW08} and~\cite[Eq.(3)]{Lin19} respectively. 
	Hence $\{H_{n,1}\}_{n=0}^{\infty}$ and $\{H_{n,2}\}_{n=0}^{\infty}$  are both SM sequences. 
	(Actually the second equation in~\eqref{explicit} yields that $\{H_{n-1,2}\}_{n=0}^{\infty}$ is an SM sequence 
	and the shifted sequence of an SM sequence is still an SM sequence.)
	It is natural to consider the SM property of $\{H_{n,d}\}_{n=0}^{\infty}$ for $d\ge 3$.
	
	The polynomial sequence $\{f_n(x)\}_{n\ge 0}$ is called {\it $x$-log-convex}
	if $f_{n-1}(x)f_{n+1}(x)-f_n(x)^2$ has nonnegative coefficients for all $n\geq 1$.
	Clearly the $x$-log-convexity of $H_{n,d}(x)$ implies the log-convexity of $H_{n,d}$.
	The $x$-log-convexity of the Narayana polynomial $H_{n,2}(x)$ was conjectured by Liu and Wang~\cite{LW07},
	and solved by Chen, Yang and Wang~\cite{CYW10} and Zhu~\cite{Zhu13}.
	Zhu~\cite{Zhu24} studied $x$-log-convexity of the row-generating polynomials of certain matrices.
	An interesting problem is to show the $x$-log-convexity of $H_{n,d}(x)$ for $d\ge 3$.

	Yu~\cite{Yu09} strengthened Theorem~\ref{ray} (i) on Pascal's triangle.
	
	\begin{thm}[{\cite[Section 2]{Yu09}}]\label{Conj-Yu}
		Let $n_0,k_0,p,q$ be integers such that $n_0\ge k_0\ge 0,q>p>0,$ and $k_0<q$.
		Define $C_i=\binom{n_0+ip}{k_0+iq}$, $i=0,1,2,\dots.$
		Then the polynomial $\sum_{i\ge 0} C_ix^i$ has only real zeros.
	\end{thm}

	It might be of interest to ask whether these results can be generalized to the $d$-Hoggatt triangle for $d\ge 2$.
	
	\section*{Acknowledgement}
	The authors greatly appreciate the anonymous referees for their careful reading and valuable suggestions.
	The authors would also like to thank Yaling Wang 
	for helpful comments.
	This work was partially supported by the National Natural Science Foundation of China (No. 12201100).


\end{document}